\documentclass[sn-mathphys,Numbered]{sn-jnl}

\usepackage{graphicx}%
\usepackage{multirow}%
\usepackage{amsmath,amssymb,amsfonts}%
\usepackage{amsthm}%
\usepackage{mathrsfs}%
\usepackage[title]{appendix}%
\usepackage{xcolor}%
\usepackage{textcomp}%
\usepackage{manyfoot}%
\usepackage{enumerate}

\usepackage{mathdots}  

\usepackage{tikz}
\colorlet{myblue}{blue!60!black}
\definecolor{bleudefrance}{rgb}{0.19, 0.55, 0.91}

\usepackage{color}
\usepackage[normalem]{ulem}

\newcommand{\R}{\mathbb R}

\newcommand{\Z}{\mathbb{Z}}
\newcommand{\N}{\mathbb{N}}

\makeatletter
\def\moverlay{\mathpalette\mov@rlay}
\def\mov@rlay#1#2{\leavevmode\vtop{%
   \baselineskip\z@skip \lineskiplimit-\maxdimen
   \ialign{\hfil$\m@th#1##$\hfil\cr#2\crcr}}}
\newcommand{\charfusion}[3][\mathord]{
    #1{\ifx#1\mathop\vphantom{#2}\fi
        \mathpalette\mov@rlay{#2\cr#3}
      }
    \ifx#1\mathop\expandafter\displaylimits\fi}
\makeatother
\newcommand{\cupdot}{\charfusion[\mathbin]{\cup}{\cdot}}

\theoremstyle{thmstyleone}%
\newtheorem{theorem}{Theorem}
\newtheorem{proposition}[theorem]{Proposition}%
\newtheorem{corollary}[theorem]{Corollary}
\newtheorem{lemma}[theorem]{Lemma}
\newtheorem{Question}[theorem]{Question}

\newtheorem{remark}[theorem]{Remark}%

\theoremstyle{thmstyletwo}%

\theoremstyle{thmstylethree}%

\begin{document}

\title[Cube tilings with linear constraints]{Cube tilings with linear constraints}


\author*[1]{\fnm{Dae Gwan} \sur{Lee}}\email{daegwans@gmail.com, daegwan@kumoh.ac.kr} 

\author[2]{\fnm{G\"otz E.} \sur{Pfander}}\email{pfander@ku.de} 

\author[3]{\fnm{David} \sur{Walnut}}\email{dwalnut@gmu.edu} 

\affil*[1]{\orgdiv{Department of Mathematics and Big Data Science}, \orgname{Kumoh National Institute of Technology}, \orgaddress{\street{Daehak-ro~61}, \city{Gumi}, \state{Gyeongsangbuk-do} \postcode{39177}, \country{Republic of Korea (South Korea)}}}

\affil[2]{\orgdiv{Mathematical Institute for Machine Learning and Data Science (MIDS)}, \orgname{Katholische Universit\"at Eichst\"att--Ingolstadt}, \orgaddress{\street{Hohe-Schul-Str.~5}, \postcode{85049} \city{Ingolstadt}, \country{Germany}}}

\affil[3]{\orgdiv{Department of Mathematical Sciences}, \orgname{George Mason University}, \orgaddress{\street{4400~University Drive}, \city{Fairfax}, \state{VA} \postcode{22030}, \country{USA}}}

%
%


\abstract{We consider tilings $(\mathcal{Q},\Phi)$ of $\mathbb{R}^d$ where $\mathcal{Q}$ is the $d$-dimensional unit cube and the set of translations $\Phi$ is constrained to lie in a pre-determined lattice $A \mathbb{Z}^d$ in $\mathbb{R}^d$. 
We provide a full characterization of matrices $A$ for which such cube tilings exist when $\Phi$ is a sublattice of $A\mathbb{Z}^d$ with any $d \in \N$ or a generic subset of $A\mathbb{Z}^d$ with $d\leq 7$.   
As a direct consequence of our results, we obtain a criterion for the existence of linearly constrained frequency sets, that is, $\Phi \subseteq A\mathbb{Z}^d$, such that the respective set of complex exponential functions $\mathcal{E} (\Phi)$ is an orthogonal Fourier basis for the space of square integrable functions supported on a parallelepiped $B\mathcal{Q}$, where $A, B \in \R^{d \times d}$ are nonsingular matrices given a priori. Similarly constructed Riesz bases are considered in a companion paper \cite{LPW24-second-RB}.}




\maketitle

\section{Introduction and main results}
\label{sec:intro}

For a measurable set $S \subseteq \R^d$ and a discrete set $\Phi \subseteq \R^d$, we say that 
$(S, \Phi)$ is a \emph{tiling pair} (or simply a \emph{tiling}) if the sets $S {+} \varphi$, $\varphi \in \Phi$, form a partition of $\R^d$, i.e., $\bigcup_{\varphi \in \Phi} S {+} \varphi = \R^d$ and 
$( S{+}\varphi ) \cap ( S{+}\varphi' ) =\emptyset$ for all $\varphi\neq \varphi'$ in $\Phi$.  
We will denote the unit cube in $\R^d$ by $\mathcal{Q} = \mathcal{Q}^{(d)} = [0,1)^d$, and we particularly refer to tilings of the form $(\mathcal{Q},\Phi)$ as \emph{cube tilings}. For a cube tiling   $(\mathcal{Q} ,\Phi)$  and $v\in \Z^d$,
we denote by $\phi_v\in \Phi$ the unique element in $\Phi$ with $v\in \mathcal{Q}^{(d)}+ \phi_v$.   

This paper addresses cube tilings $(\mathcal{Q},\Phi)$ with the constraint that $\Phi \subseteq A \Z^d$ for a predetermined full rank matrix $A\in\R^{d\times d}$.  

\begin{theorem}\label{thm:main}
For a nonsingular matrix $A \in \R^{d \times d}$, the following are equivalent:
\begin{enumerate}[i.]
     \item\label{list:thm1part1} there exists a lattice $\Lambda \subseteq A\Z^d$ such that $(\mathcal{Q}, \Lambda)$ is a tiling;
     \item\label{list:thm1part2} there exist a permutation matrix $P \in \{0,1\}^{d \times d}$ and a nonsingular integer matrix $R \in \Z^{d \times d}$ such that $PAR$ is unitriangular.\footnote{A triangular matrix is called unitriangular if its diagonal entries are all equal to $1$. The condition of $PAR$ being unitriangular can be replaced with $PAR$ being \emph{lower unitriangular} or $PAR$ being \emph{upper unitriangular} (cf.~Remark~\ref{rmk:lower-upper-unitriangular-matrices}).}
 \end{enumerate} 
 For $d\leq 7$, the statements above are also equivalent to:
 \begin{enumerate}[i.]
 \setcounter{enumi}{2}
     \item\label{list:thm1part3} there exists a set $\Phi \subseteq A\Z^d$ such that $(\mathcal{Q}, \Phi)$ is a tiling.
\end{enumerate}   
\end{theorem}

The equivalence of $\ref{list:thm1part1}$ and $\ref{list:thm1part2}$ in Theorem~\ref{thm:main} is a  consequence of the lattice case, that is, $(b)$ of the following fundamental result on cube tilings which is  discussed in some detail in Section~\ref{sec:background}.

\begin{theorem}[Minkowski \cite{Mi1907}, Haj{\'o}s \cite{Ha42}, Keller \cite{Keller1930}, Perron \cite{Pe40a,Pe40b}, Brakensiek et al.~\cite{BHMN20}]
\label{thm:Minkowski-Hajos} 
If $(\mathcal{Q}, \Phi)$ is a cube tiling with 
(a) $d\leq 7$ or (b) $\Phi$ being a lattice, then exists $v,w\in \{0,1\}^d$ so that $Q+{\phi_v}$ and $Q+{\phi_w}$  
are cube twins, that is, both cubes  share a $(d{-}1)$-dimensional face.   
\end{theorem}

In the non-lattice case, the restriction of $d \leq 7$ is necessary to ensure the existence of cube twins \cite{LS92,Ma02}. 
Under   this restriction, Theorem~\ref{thm:Minkowski-Hajos} can be used to construct for our purpose useful lattice tilings from non-lattice tilings.

\begin{theorem}\label{thm:d-up-to-7-construct-lattice}
    For $d\leq 7$ and a tiling $(\mathcal{Q} ,\Phi)$ of $\R^d$ with $\phi_{(0,\ldots,0)}=(0,\ldots,0)$, there exist \\ $v_1,\ldots,v_d,w_1,\ldots,w_d\in \{0,1\}^d$ 
    such that  
    \begin{align}\label{eqn:existsLatticea}
        \Big( \mathcal{Q} ,(\phi_{v_1}-\phi_{w_1})\Z+\ldots + (\phi_{v_d}-\phi_{w_d})\Z \Big)
    \end{align} is a tiling pair. 
    Moreover, the vectors can be chosen so that there exists a reordering of the coordinates which leads to 
    \begin{align}\label{eqn:existsLatticeb}
        v_j-w_j\in \{0\}^{j-1}\times\{1\}\times\{0,1\}^{d-j}\quad \text{for} \quad j=1,\ldots,d.
    \end{align}    
\end{theorem}

Note that  ``$\ref{list:thm1part1}$ implies $\ref{list:thm1part3}$'' in Theorem~\ref{thm:main} in the case $d\leq 7$ follows directly from Theorem~\ref{thm:d-up-to-7-construct-lattice} as $\Phi\subseteq A\Z^d$ implies $(\phi_{v_1}-\phi_{w_1})\Z+\ldots + (\phi_{v_d}-\phi_{w_d})\Z\subseteq A\Z^d$. Further, the restriction $\phi_{(0,\ldots,0)}=(0,\ldots,0)$ is not consequential as we can always achieve this by replacing $\Phi$ with $\Phi-\phi$ for any choice of $\phi\in\Phi$.

  \begin{remark}\label{rmk:d-leq-7-cannot-be-improved-continued}
  \rm
  Note that \eqref{eqn:existsLatticea} and \eqref{eqn:existsLatticeb} in 
Theorem~\ref{thm:d-up-to-7-construct-lattice} imply together that $\mathcal{Q}+\phi_{v_d}$ and $\mathcal{Q}+\phi_{w_d}$ are twin cubes, as  \eqref{eqn:existsLatticeb} implies that $\phi_{v_d}-\phi_{w_d}=e_i$ for some $i \in \{ 1,\ldots,d \}$. Since not every tiling for $d\geq 8$ contains twin cubes, Theorem~\ref{thm:d-up-to-7-construct-lattice} does not hold for $d\geq 8$.

Nonetheless, the conclusion containing only \eqref{eqn:existsLatticea} may hold for $d\geq 8$.  It is worth noting, that the examples that show that $d\leq 7$ in Theorem~\ref{thm:Minkowski-Hajos}, part $(a)$, cannot be improved, i.e., the examples provided for $d=8$ by Mackey in \cite{Ma02} and by Lagarias and Shor for $d=10,12$ in \cite{LS92},  are likely not counterexamples to the conclusion \eqref{eqn:existsLatticea}.  Indeed, the constructed tilings satisfy $\Phi\subseteq \frac 1 2 \Z^d$, and the standard lattice tiling $(\mathcal{Q}, \Z^d)$ satisfies $\Z^d \subseteq \frac{1}{2} \Z^d$. Certainly, it is still possible that the constructed $\Phi\subseteq A \Z^d$ for some non-trivial full rank matrix $A$.
\end{remark}
Remark~\ref{rmk:d-leq-7-cannot-be-improved-continued} identifies the following open  problems.
\begin{Question} ~
\begin{enumerate}[i.]
    \item Given a tiling pair $(\mathcal{Q},\Phi)$ for $\R^d$  with $\Phi\subseteq A\Z^d$ and  $d\geq 8$, does there exist a lattice $\Lambda\subseteq A\Z^d$ so that $(\mathcal{Q},\Lambda)$ is a tiling pair as well?  
    \item For those $d$ for which the question above is answered positively, the missing direction ``$\ref{list:thm1part3}$ implies $\ref{list:thm1part1}$'' in Theorem~\ref{thm:main} holds. 
    But if the above is not confirmed for all $d$, can the restriction $d\leq 7$ in Theorem~\ref{thm:main} be nonetheless relaxed, or omitted entirely? 
\end{enumerate}
    \end{Question}

    \medskip

Note that if $(\mathcal{Q}, \Lambda)$ is a lattice tiling with $\Lambda\subseteq A\Z^d$ or if $(\mathcal{Q}, \Phi)$ is a generic tiling with $\Phi\subseteq A\Z^d$ and $d\leq 7$, then Theorem~\ref{thm:main} implies $1=\det PAR=\det A \, \det R$ and, since $R$ is an integer matrix, we have $|\det A|=\frac 1 N$ for some $N\in\N$. In an early  arXiv version of this article, we  asked the question whether the conclusion $|\det A|=\frac 1 N$ for $N\in\N$ holds in the non-lattice case also for $d\geq 8$. Mihalis Kolountzakis confirmed this by showing
\begin{theorem}[Kolountzakis]\label{thm:Kolountzakis}
    If $(S,\Phi)$ is a tiling pair with $\Phi\subseteq \Z^d$, then the volume of $S \subseteq \R^d$ is a natural number. Consequently, if $(\mathcal Q,\Phi)$ is a tiling pair with $\Phi\subseteq A\Z^d$, then $|\det A|=1/N$ for some $N\in\N$.
\end{theorem}
 Kolountzakis'  proof of Theorem~\ref{thm:Kolountzakis} is included  in Section~\ref{sec:Kolountzakis}. 



\medskip

Cube tilings are closely related to orthogonal exponential bases for spaces of square integrable functions.  
In fact, $(\mathcal{Q},\Gamma)$ is a tiling pair if and only if 
it is a so called \emph{spectral pair}, that is, $\mathcal{E}(\Gamma) = \{e^{2\pi i \gamma \cdot (\cdot)} : \gamma \in\Gamma\}$ is an orthogonal bases for  $L^2(\mathcal{Q})$ \cite{JP99,Ko00,LRW00}. In this setting, we refer to the discrete set  $\Gamma$ as  \emph{frequency set} or  \emph{spectrum}, and the set $\mathcal{Q}$ is called a \emph{domain}. Note that this equivalence is easily verified if $\Gamma$ is a lattice, but is highly nontrivial in the non-lattice case.



The relationship of tiling and spectral pairs, together with a change of variable argument, provides the following consequence of Theorem~\ref{thm:main}.

\begin{corollary}\label{cor:main-application-to-expRB}
For nonsingular matrices $A,B \in \R^{d \times d}$, the following are equivalent:
\begin{enumerate}[i.]
     \item\label{list:thm1part12} there exists a lattice $\Lambda \subseteq A\Z^d$ with $\mathcal E(\Lambda)$ being an orthogonal basis for $L^2(B\mathcal{Q})$;
     \item\label{list:thm1part22} there exist a permutation matrix $P \in \{0,1\}^{d \times d}$ and a nonsingular integer matrix $R \in \Z^{d \times d}$ with $PB^TAR$ being unitriangular; 
     \item[{\romannumeral 2${.}^\prime$\!}]\label{list:thm1part22-prime} 
     there exist a nonsingular integer matrix $R \in \Z^{d \times d}$ and a unitriangular matrix $G \in \R^{d \times d}$ with $B\mathcal{Q} = A^{-T} R^{-1} G \mathcal{Q}$.   
\end{enumerate} 
The unitriangularity condition in 
({\romannumeral 2}) and ({\romannumeral 2${}^\prime$\!}) can be replaced with lower unitriangularity or upper unitriangularity. 

For $d\leq 7$, the statements above are also equivalent to: 
 \begin{enumerate}[i.]
 \setcounter{enumi}{2}
     \item\label{list:thm1part32} there exists a set $\Phi \subseteq A\Z^d$ with $\mathcal E(\Phi)$ being an orthogonal basis for $L^2(BQ)$.
\end{enumerate}   
\end{corollary}


\medskip

The paper is structured as follows. 
Section~\ref{sec:background} provides some background on cube tilings and on exponential bases. Before proving our results in later sections, we consider the case $d=4$
in Section~\ref{sec:examples} to enable  the reader to build some intuition on the problem at hand. 
Sections~\ref{sec:proof-thm-main}, \ref{sec:proof-thm-d-up-to-7-construct-lattice} and \ref{sec:proof-cor-main-application-to-expRB} are devoted to the proof of Theorem~\ref{thm:main}, Theorem~\ref{thm:d-up-to-7-construct-lattice} and Corollary~\ref{cor:main-application-to-expRB}, respectively. Section~\ref{sec:Kolountzakis} contains Kolountzakis' proof of Theorem~\ref{thm:Kolountzakis}.

\section{Background}
\label{sec:background}

Before proceeding to prove our results, we give some related background on cube tilings and complex exponential bases.

\medskip
\noindent
{\it Cube tilings}. \

The study of tilings, particularly within the context of the classification of planar wallpaper patterns, began in the early 19th century. This led to the development of a rich theory on symmetry groups and further contributed to advancements in, for example, crystallography. 
Due to the simple structure of cubes, cube tilings are well suited to provide insights into the increasingly complex geometry of high dimensional Euclidean space. 

The canonical cube tiling of $\R^d$ is given by $(\mathcal{Q},\Z^d)$.
A wide class of cube tilings can be obtained by shifting some rows or columns in the tiling $(\mathcal{Q},\Z^d)$:
  
For $d=2$, a global shift can be combined with  either shifting each column independently of each other and obtain $\Phi=\{(n+c,\alpha(n)+m) : n,m\in\Z \}$ or each row independently of each other and obtain $\Phi=\{(\alpha(m)+n,m+c) : n,m\in\Z \}$ with $c\in\R$ and $\alpha:\Z\to\R$. It is easily checked that these are all the cube tilings of $\R^2$.

For $d=3$, we can obtain a cube tiling where shifts in each coordinate direction have taken place, see Figure~\ref{fig:Nir}. But there are cube tilings which cannot be obtained by shifts of subsets into coordinate directions.  As $d$ increases, the situation becomes increasingly cumbersome, see, e.g., \cite{Zo05}.

In the early 20th century, the interest in cube tilings was enhanced by the Minkowski and the Keller conjecture. Minkowski observed that in low dimensions, any lattice cube tiling contains twin cubes, that is, a pair of cubes in a $d$-dimensional cube tiling that share a $(d{-}1)$-dimensional face. 
Minkowski consequently conjectured that this result holds for all $d$ \cite{Mi1907}.
This conjecture  was  confirmed later by Haj{\'o}s \cite{Ha42}. In the mean time, Keller\footnote{In fact, Keller stated a stronger statement that every cube tiling has a column of cubes all meeting face-to-face.} had extended Minkowski's conjecture by postulating that the conclusion holds also for any non-lattice  tiling pair $(\mathcal{Q},\Phi)$ \cite{Keller1930}. Perron confirmed this for $d\leq 6$ \cite{Pe40a,Pe40b}.
However, Keller's conjecture was shown to be false for $d\geq 10$ by Lagarias and Shor \cite{LS92} and for $d \geq 8$ by Mackey \cite{Ma02}. The last open case $d=7$ was settled only recently with a computer-assisted proof that involved a dataset of 200 gigabytes \cite{BHMN20}. In summary, Keller's conjecture holds if and only if $d\leq 7$.  

Certainly, due to the simple structure of a cube, the study of cube tilings is not representative of the study of tilings with sets. To illustrate this, let us point to a recent paper of Greenfeld and Tao where a set $\Sigma \subset \R^d$ is constructed to disprove the periodic tiling conjecture \cite{GT22}. That is, for  $\Sigma \subset \R^d$ there exists a discrete set $\Phi \subset \R^d$ such that $(\Sigma,\Phi)$ indeed tiles $\R^d$ a.e.\footnote{Here, the tiling conditions are given up to a set of Lebesgue measure 0.} but no periodic tiling set exists in the sense that for every lattice $\Lambda \subset \R^d$ and for any choice of finitely many points $a_1,\ldots,a_n \in \R^d$, $(\Sigma, \Lambda \cup (\Lambda{+}a_1) \cup \ldots \cup (\Lambda{+}a_n) )$ does not tile $\R^d$ a.e. Note that the dimension $d$ here is not explicitly given, but described as enormous.


For additional reading on cube and generic tilings, see e.g., 
\cite{GT21,GT23,GK23,Ki13,Ki20,MOP13,KM06,GS87,LW97}.

\bigskip

\noindent
{\it Fourier and exponential bases}. \

Fourier bases, composed of complex exponentials with frequencies that are integer multiples of a base frequency, are used extensively in various fields of science, mathematics and engineering.  
Exponential bases are a generalization of the Fourier bases; they allow for a flexible choice of frequencies \cite{Yo01}. 
Below we provide a short review on some to this paper related problems and results Fourier and exponential bases.

Fuglede's conjecture \cite{Fu74} states that if $S \subseteq \R^d$ is a finite measure set, then the space $L^2(S)$ has an orthogonal basis $\mathcal{E}(\Gamma)$ with some $\Gamma \subseteq \R^d$
if and only if there exists a discrete set $\Omega \subseteq \R^d$ such that $(S, \Omega)$ is a tiling pair for $\R^d$ up to a measure zero set. 
This conjecture is shown to be false for $d \geq 3$, but it remains open for $d = 1,2$.
However, the conjecture is known to be true in some special cases (see \cite{LM19} and the references therein).
For example, if $\Gamma$ is a lattice of $\R^d$ \cite{Fu74}, or if $S \subseteq \R^d$ is a convex set of finite positive measure \cite{LM19}, then the conjecture holds in all dimensions.
In particular, if $S$ is the Euclidean unit ball in $\R^d$ with $d \geq 2$, then no exponential orthogonal basis exists for $L^2(S)$ \cite{IKP99}, which is in contrast with the $1$-dimensional case where $\mathcal{E}(\frac{1}{2}\Z)$ is clearly an orthogonal basis for $L^2 [-1,1]$. 
Also, for any discrete set $\Gamma \subseteq \R^d$, the system $\mathcal{E}(\Gamma)$ is an orthogonal basis for $L^2(\mathcal{Q})$ if and only if $(\mathcal{Q},\Gamma)$ is a tiling of $\R^d$; in short, $(\mathcal{Q},\Gamma)$ is a spectral pair if and only if it is a tiling pair \cite{JP99}. 
Moreover, the cube $\mathcal{Q}$ in this statement can be replaced with any open set of measure $1$ \cite{Ko00}. 
When $S = A \mathcal{Q}$ with $A \in \mathrm{GL} (d,\R)$, we have that for any discrete set $\Gamma \subseteq \R^d$, the system $\mathcal{E}(\Gamma)$ is an orthogonal basis for $L^2(S)$ if and only if 
 $(A^{-T} \mathcal{Q}, \Gamma)$ is a tiling of $\R^d$ \cite{LRW00}.




In the case of Riesz bases which is a natural generalization of orthogonal bases, the following results are worth mentioning in our context. 
For any interval $I \subseteq [0,1)$, there exists a set $\Gamma \subseteq \Z$ such that $\mathcal{E}(\Gamma)$ is a Riesz basis for $L^2(I)$ \cite{Se95}. If $S$ is a finite union of disjoint intervals in $[0,1)$, then there exists a set $\Gamma \subseteq \Z$ such that $\mathcal{E}(\Gamma)$ is a Riesz basis for $L^2(S)$ \cite{KN15}.
However, there are only a few classes of sets for which exponential Riesz bases are known to exist, see e.g., \cite{CC18,CL22,DL19,Ko15,KN15,Le22,LPW23,PRW21}.
Although most of the sets in $\R^d$ are generally believed to have exponential Riesz bases, there is a specific bounded measurable set $S \subseteq \R$ such that no exponential Riesz basis exists for $L^2(S)$ \cite{KNO21}.

\section{Construction of a lattice tiling from a generic tiling in $d=4$}
\label{sec:examples}

In this section,  we consider the case $d=4$ and show that ``$\ref{list:thm1part3}$ implies  $\ref{list:thm1part1}$''  in Theorem~\ref{thm:main} by first principles, mainly relying on Keller's theorem  \cite{Keller1930}:
\begin{theorem}\label{thm:keller}
If $(Q,\Phi)$ is a tiling pair, then for every distinct elements $\phi,\phi'\in\Phi$ their difference $\phi-\phi'$ has a nonzero integer coordinate. 
\end{theorem}

In this section, for brevity of notation, we shall write elements of $\R^4$ as row vectors.

\medskip
Let  now $([0,1)^4,\Phi)$ be a generic tiling pair with $(0,0,0,0)\in \Phi$. 
As $4 \leq 7$, Keller's conjecture for dimension $d \leq 7$ (see Theorem~\ref{thm:Minkowski-Hajos} and Section~\ref{sec:background}) shows that there exist $\phi,\psi \in \Phi$ with $\phi-\psi=e_j$ for some $j\in\{1,2,3,4\}$. Without loss of generality, let us assume $j=d=4$. 

We now choose $\gamma,\xi,\eta,\alpha,\beta,\tau\in\Phi$ with
\begin{align*}
\begin{array}{rrrrrrrrrrr}
 (1,0,0,0)&\in &[0,1)^4&{+}& \gamma, \quad \gamma &=(&1-\gamma_1,&-\gamma_2,&-\gamma_3,&-\gamma_4&),\\
   (0,1,0,0)&\in &[0,1)^4&{+}& \xi, \quad \xi&=(&-\xi_1,&1-\xi_2,&-\xi_3,&-\xi_4&),\\
  (0,0,1,0)&\in &[0,1)^4&{+}& \eta, \quad  \eta&=(&-\eta_1,&-\eta_2,&1-\eta_3,&-\eta_4&),\\
   (1,1,0,0)&\in & [0,1)^4&{+}& \alpha, \quad \alpha &=(&1-\alpha_1,&1-\alpha_2,&-\alpha_3,&-\alpha_4&),\\
   (1,0,1,0)&\in & [0,1)^4&{+}& \beta, \quad \beta&=(&1-\beta_1,&-\beta_2,&1-\beta_3,&-\beta_4&),\\
  (0,1,1,0)&\in & [0,1)^4&{+}& \tau, \quad  \tau&=(&-\tau_1,&1-\tau_2,&1-\tau_3,&-\tau_4&),
\end{array} 
\end{align*}
where $\gamma_i,\xi_i,\eta_i,\alpha_i,\beta_i,\tau_i \in [0,1)$ for $i=1,2,3,4$.  Applying Keller's theorem, Theorem~\ref{thm:keller}, to  compare $(0,0,0,0)\in\Phi$ with the vectors above implies immediately
\begin{align*}
    & \gamma_1=0,\ \xi_2=0, \ \eta_3=0, \\  
    & \alpha_1=0\text{ or }\alpha_2=0,  \\   
    & \beta_1=0\text{ or }\beta_3=0, \\  
    &\tau_2=0\text{ or }\tau_3=0.
\end{align*}
The `or' conjugation implies that \emph{a-priori} we have to consider 8 cases.  But these reduce to the following two.

\vspace{.5cm}

\def\cube at (#1,#2,#3){\pgfmathsetmacro{\cubex}{1}
\pgfmathsetmacro{\cubey}{1}
\pgfmathsetmacro{\cubez}{-1}
\begin{scope}[shift={(#1,#2,#3)}]

\draw[black] (0,0,\cubez) -- ++(\cubex,0,0) -- ++(0,\cubey,0) -- ++(-\cubex,0,0) -- cycle;
\draw[black] (0,0,\cubez)-- ++(\cubex,0,0)-- ++(0,0,-\cubez) -- ++(-\cubex,0,0) -- cycle;
\draw[black] (0,0,\cubez) -- ++(0,\cubey,0) -- ++(0,0,-\cubez) -- ++(0,-\cubey,0) -- cycle;

\draw[black,fill=white!80] (\cubex,\cubey,0) -- ++(-\cubex,0,0) -- ++(0,-\cubey,0) -- ++(\cubex,0,0) -- cycle;
\draw[black,fill=white!80] (\cubex,\cubey,0)-- ++(-\cubex,0,0)-- ++(0,0,\cubez) -- ++(\cubex,0,0) -- cycle;
\draw[black,fill=white!80] (\cubex,\cubey,0) -- ++(0,-\cubey,0) -- ++(0,0,\cubez) -- ++(0,\cubey,0) -- cycle;

\end{scope}}
\def\cubeColor at (#1,#2,#3){\pgfmathsetmacro{\cubex}{1}
\pgfmathsetmacro{\cubey}{1}
\pgfmathsetmacro{\cubez}{-1}
\begin{scope}[shift={(#1,#2,#3)}]

\draw[blue] (0,0,\cubez) -- ++(\cubex,0,0) -- ++(0,\cubey,0) -- ++(-\cubex,0,0) -- cycle;
\draw[blue] (0,0,\cubez)-- ++(\cubex,0,0)-- ++(0,0,-\cubez) -- ++(-\cubex,0,0) -- cycle;
\draw[blue] (0,0,\cubez) -- ++(0,\cubey,0) -- ++(0,0,-\cubez) -- ++(0,-\cubey,0) -- cycle;

\draw[blue,fill=white!80] (\cubex,\cubey,0) -- ++(-\cubex,0,0) -- ++(0,-\cubey,0) -- ++(\cubex,0,0) -- cycle;
\draw[blue,fill=white!80] (\cubex,\cubey,0)-- ++(-\cubex,0,0)-- ++(0,0,\cubez) -- ++(\cubex,0,0) -- cycle;
\draw[blue,fill=white!80] (\cubex,\cubey,0) -- ++(0,-\cubey,0) -- ++(0,0,\cubez) -- ++(0,\cubey,0) -- cycle;

\end{scope}}

\begin{figure}[t]
\begin{center}
\begin{tikzpicture}[fill opacity=0.9, scale=2]
\newcommand{\gaaa} {0.3};
\newcommand{\xiii} {0.5};
\newcommand{\ettt} {0.7};
\cube at (0,0,0); 
\cube at (0,1,-\xiii); 
\cube at (1,-\gaaa,0); 
\cube at (-\ettt,0,1); 
\cube at (1,1-\gaaa,0); 
\cube at (1-\ettt,0,1);
\cube at (0,1,1-\xiii);

\draw[black] (1,1,0) -- (1,2-\gaaa,0) -- (1,2-\gaaa,-1);
\draw[white!90] (1,1,0) -- (1,1,-\xiii) -- (1,2-\gaaa,-\xiii);
\draw[black] (1,2-\gaaa,0) -- (2,2-\gaaa,0);


\fill[red] (0,0,0) circle (0.2mm); \node[above] at (0.1,-.05,0) {\tiny (\!0\!,0\!,0\!,0\!)};


\end{tikzpicture}
    \end{center}
    \caption{
    We display the projection of  the cubes $[0,1)^4$, $[0,1)^4+\gamma$, $[0,1)^4+\xi$, $[0,1)^4+\eta$, $[0,1)^4+\alpha$, $[0,1)^4+\beta$, $[0,1)^4+\tau$ along $e_4$. This scenario in the case of $d=3$ is achieved, for example, by starting with the canonical tiling  $([0,1)^4,\Z^4)$ and then shifting disjoint columns -- one in $x$ direction, one in $y$ direction, and one in $z$ direction -- that neighbor the origin in $x$ neighboring by a non integer.
}
\label{fig:Nir}
\end{figure}
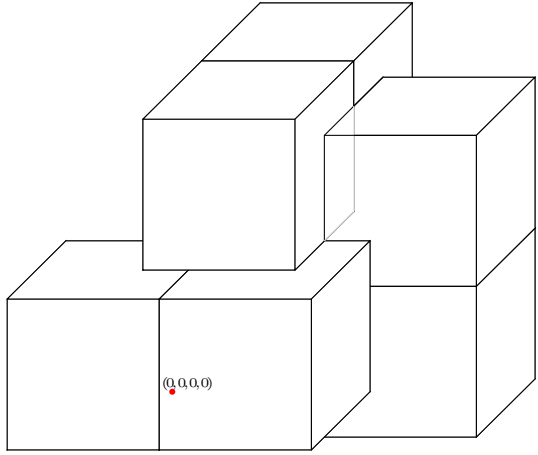

\noindent \emph{Case 1.} In at least one of the first three coordinates, three of the six vectors above contain a $1$.

Without loss of generality, suppose this is the case for the first coordinate so that $\alpha_1=\beta_1=0$.
We define the lattice 
$$
\Lambda= (\gamma-0) \Z + (\alpha-0) \Z + (\beta-0) \Z+ (\psi-\phi)\Z=\gamma \Z + \alpha \Z + \beta \Z+ e_4\Z .
$$

To show that  $([0,1)^4,\Lambda)$ is  a tiling pair, we start with some additional observations based on Keller's theorem. 
As  $\gamma-\alpha$ must have a non zero integer component, we must have $\gamma_2=\alpha_2$. Similarly  $\gamma-\beta$ having an integer component implies $\gamma_3=\beta_3$, and the same for $\alpha-\beta$ implies $\alpha_2=\beta_2$ or $\alpha_3=\beta_3$.  

If $\alpha_2=\beta_2$, then
\begin{align*}
\begin{array}{clccccc}
    \gamma&=(&1,&-\gamma_2,&-\gamma_3,&-\gamma_4&),\\
    \alpha&=(&1,&1-\gamma_2,&-\alpha_3,&-\alpha_4&),\\
    \beta&=(&1,&-\gamma_2,&1-\gamma_3,&-\beta_4&), 
\end{array}
\end{align*}
and if $\alpha_3=\beta_3$
\begin{align*}
\begin{array}{clccccc}
    \gamma&=(&1,&-\gamma_2,&-\gamma_3,&-\gamma_4&),\\
    \alpha&=(&1,&1-\gamma_2,&-\gamma_3,&-\alpha_4&),\\
    \beta&=(&1,&-\beta_2,&1-\gamma_3,&-\beta_4&).
\end{array} 
\end{align*}

Now suppose that $\big( [0,1)^4{+}\lambda \big) \cap \big( [0,1)^4{+}\lambda' \big) \neq \emptyset$ for some $\lambda,\lambda'\in\Lambda$. By simple translation, we can assume that $\lambda'=0$. Since $\big( [0,1)^4{+}\lambda \big) \cap [0,1)^4 \neq \emptyset$ we have
\begin{align*}
   \lambda&= k\gamma+\ell \alpha+m\beta+ne_4\in (-1,1)^4. 
   \end{align*}
  If $\alpha_2=\beta_2$, then
   \begin{align*}
   \begin{array}{llll}
       k+\ell+m &\in (-1,1) &\text{ and therefore } &k+\ell+m=0,\\ 
     -\gamma_2k+(1-\gamma_2)\ell-\gamma_2m=\ell&\in (-1,1) &\text{ and therefore } & \ell=0, \ k=-m,\\
      -\gamma_3 k -\alpha_3 \ell+(1-\gamma_3)m=m&\in (-1,1)&\text{ and therefore } &k=m=0,\\ 
      -\gamma_4k-\alpha_4\ell-\beta_4m +n =n&\in (-1,1)&\text{ and therefore } &n=0,\\ 
   \end{array}
   \end{align*}
   and, hence, $\lambda=0$. The case $\alpha_3=\beta_3$ follows identically and we conclude that all multidimensional cubes  $[0,1)^4{+}\lambda$, $\lambda\in\Lambda$, are disjoint.  
   Since
   \begin{align*}
        \det&\left(\begin{array}{cccc}
     1 &-\gamma_2 &-\gamma_3 &-\gamma_4\\
     1 &1-\gamma_2&-\alpha_3&-\alpha_4\\
     1&-\gamma_2&1-\gamma_3&-\beta_4\\
     0&0&0&1
\end{array} \right)\\&=(1-\gamma_2)(1-\gamma_3)+\gamma_2\alpha_3 +\gamma_2\gamma_3 + \gamma_3(1-\gamma_2)-\gamma_2\alpha_3+ \gamma_2(1-\gamma_3)\\&
=1 -\gamma_2-\gamma_3 +\gamma_2\gamma_3 +\gamma_2\gamma_3 + \gamma_3-\gamma_3\gamma_2+ \gamma_2-\gamma_2\gamma_3=1
    \end{align*} 
     the lattice $\Lambda$ has density 1, and we conclude using Lemma~\ref{lem:equivalencetiling} below that $([0,1)^4,\Lambda)$ is a tiling pair.

Let us now consider the complementary case: 
\vspace{.5cm}

\noindent \emph{Case 2.} In each of the first three coordinates, exactly two of the six vectors above contain a $1$.

That is,
\begin{align*}
    \alpha_1=0 \text{ or } \alpha_2=0, \ \ \beta_1=0 \text{ or } \beta_3=0, \ \ \tau_2=0 \text{ or } \tau_3=0.
\end{align*}
Without loss of generality, we assume $\alpha_1=0$, $\beta_3=0$, and $\tau_2=0$, and 
$\alpha_2\neq 0$, $\beta_1\neq 0$, and $\tau_3\neq 0$ (as else, we would be in \emph{Case 1.}) With the non-zero elements in bold face 
\begin{align*}
\begin{array}{rlccccc}
 \gamma &=(&1,&-\gamma_2,&-\gamma_3,&-\gamma_4&),\\
   \xi&=(&-\xi_1,&1,&-\xi_3,&-\xi_4&),\\
    \eta&=(&-\eta_1,&-\eta_2,&1,&-\eta_4&),\\
     \alpha &=(&1,&1-\boldsymbol{\alpha}_2,&-\alpha_3,&-\alpha_4&),\\
    \beta&=(&1-\boldsymbol{\beta}_1,&-\beta_2,&1,&-\beta_4&),\\
   \tau&=(&-\tau_1,&1,&1-\boldsymbol{\tau}_3,&-\tau_4&). 
\end{array} 
\end{align*}
Applying Keller's theorem to $\xi-\alpha$ yields $\xi_1=0$, 
applying it to $\gamma-\beta$ gives $\gamma_3=0$, 
and applying it to $\eta-\tau$ gives $\eta_2=0$, that is, 
\begin{align*}
\begin{array}{clccccc}
 \gamma &=(&1,&-\gamma_2,&0,&-\gamma_4&),\\
   \xi&=(& 0,&1,&-\xi_3,&-\xi_4&),\\
    \eta&=(&-\eta_1,&0,&1,&-\eta_4&),\\
     \alpha &=(&1,&1-\boldsymbol{\alpha}_2,&-\alpha_3,&-\alpha_4&),\\
    \beta&=(&1-\boldsymbol{\beta}_1,&-\beta_2,&1,&-\beta_4&),\\
   \tau&=(&-\tau_1,&1,&1-\boldsymbol{\tau}_3,&-\tau_4&). 
\end{array} 
\end{align*}
Applying now Keller's theorem to $\gamma-\alpha$, $\xi-\tau$, and $\eta-\beta$, we obtain
\begin{align*}
     \gamma_2-\boldsymbol{\alpha}_2\neq 0, \quad \xi_3=\boldsymbol{\tau}_3\neq 0, \quad \eta_1=\boldsymbol{\beta}_1\neq 0.
 \end{align*} Using this and applying the theorem to $\gamma-\tau$, $\xi-\beta$, and $\eta-\alpha$, we obtain that $\tau_1=0$, $\beta_2=0$, and $\alpha_3=0$. Consequently
\begin{align*}
\begin{array}{clccccc}
 \gamma &=(&1,&-\boldsymbol{\alpha}_2,&0,&-\gamma_4&),\\
   \xi&=(& 0,&1,&-\boldsymbol{\tau}_3,&-\xi_4&),\\
    \eta&=(&-\boldsymbol{\beta}_1,&0,&1,&-\eta_4&),\\
     \alpha &=(&1,&1-\boldsymbol{\alpha}_2,&0,&-\alpha_4&),\\
    \beta&=(&1-\boldsymbol{\beta}_1,&0,&1,&-\beta_4&),\\
   \tau&=(&0,&1,&1-\boldsymbol{\tau}_3,&-\tau_4&).
\end{array} 
\end{align*}
The respective six translates of $[0,1)^4$ as well as the set $[0,1)^4$ (in the background) are depicted in Figure~\ref{fig:Nir}. We have 
\begin{align*}
\begin{array}{clccccc}
 \alpha -\gamma &=(&0,&1,&0,&-\alpha_4+\gamma_4&)\in A \Z^4,\\
    \beta-\eta&=(&1,&0,&0,&-\beta_4+\eta_4&)\in A \Z^4,\\
   \tau-\xi&=(&0,&0,&1,&-\tau_4+\xi_4&)\in A \Z^4,
\end{array} 
\end{align*}
and consequently
\begin{align*}
    \Lambda= e_4\Z + (\alpha-\gamma)\Z + (\beta-\eta)\Z + (\tau-\xi)\Z\subseteq A \Z^4.
\end{align*}
As noted above, to see that $([0,1)^4,\Lambda)$ tiles $\R^4$, it suffices to check that $(-1,1)^4 \cap \big( [0,1)^4{+}\lambda \big) =\emptyset$ for $\lambda\neq 0$.  
This follows by observing that 
\begin{align*}
    (m,\ell,n,k) +(\gamma_4-\alpha_4)\ell+(\eta_4-\beta_4)m+(\xi_4-\tau_4)n \in [0,1)^4
\end{align*}
for some $(k,\ell,m,n)\in\Z^4$ implies $(k,\ell,m,n)=(0,0,0,0)$. That $([0,1)^4,\Lambda)$ fills the space is again observed by fact that the area of a fundamental domain of the lattice is one.

\section{Proof of Theorem~\ref{thm:main}}
\label{sec:proof-thm-main}

The proof of Theorem~\ref{thm:main} utilizes the following Lemma that characterizes lattice cube tilings. It is presented in similar form in \cite{Ko98} and \cite{Zo05}. 

\begin{lemma}\label{lem:cubetiling}
For a lattice $\Lambda \subseteq \R^d$, the following are equivalent. 
\begin{enumerate}[i.]
     \item\label{list:cubetiling1} 
$(\mathcal{Q}, \Lambda)$ is a tiling of $\R^d$, that is, $\mathcal{Q}$ is a fundamental domain of the lattice $\Lambda$. 


     \item\label{list:cubetiling3} 
$\Lambda=P G \Z^d$ with a permutation matrix $P \in \{0,1\}^{d \times d}$ and a lower unitriangular matrix $G \in \R^{d \times d}$.


%

\end{enumerate}
\end{lemma}

After the following remark but prior to the proof of Theorem~\ref{thm:main}, we include a proof of this result for sake of completeness and to illustrate the use of Haj{\'o}s' Theorem, i.e., the lattice case of Theorem~\ref{thm:Minkowski-Hajos}. 

\begin{remark}\label{rmk:lower-upper-unitriangular-matrices}~

\rm
(a) The conditions in Lemma~\ref{lem:cubetiling} are formulated in terms of the lattice $\Lambda$ itself, while the conditions in Theorem~\ref{thm:Minkowski-Hajos} are given in terms of the matrix $A \in \R^{d \times d}$ of the lattice $A \Z^d$.  
In general, we have  
\[
\begin{split}
A \mathcal{Q} = B \mathcal{Q} \;\; \quad  &\Leftrightarrow \quad  A = B P \;\; 
\text{for a permutation matrix $P \in \{0,1\}^{d \times d}$} \\
A \Z^d = B \Z^d  \quad  &\Leftrightarrow \quad  A = B U \;\; 
\text{for a unimodular matrix $U \in \Z^{d \times d}$} 
\end{split}
\]
where a unimodular matrix means a square integer matrix with determinant $+1$ or $-1$. 
Writing $\Lambda = A \Z^d$, condition $\ref{list:cubetiling3}$ can be expressed as  $A = P G U$, where $U \in \Z^{d \times d}$ is a unimodular matrix.  \\
(b) Observing that 
\[ 
	\begin{pmatrix}
		0 & \cdots & 0& 1 \\
		0 & \cdots & 1 & 0 \\
		\vdots  & \iddots & \vdots  & \vdots  \\ 
		1  & \cdots & 0 & 0	  
	\end{pmatrix} ^{-1} 	
	\begin{pmatrix}
		1& 0 & \cdots & 0 \\
		s_{2,1}  & 1  & \cdots & 0  \\ 
		\vdots & \vdots & \ddots & \vdots \\
		s_{d,1}  & s_{d,2} & \cdots & 1	 
	\end{pmatrix} 
	\begin{pmatrix}
		0 & \cdots & 0& 1 \\
		0 & \cdots & 1 & 0 \\
		\vdots  & \iddots & \vdots  & \vdots  \\ 
		1  & \cdots & 0 & 0	 
	\end{pmatrix}  	  
=
	\begin{pmatrix}
		1& s_{d,d-1} & \cdots & s_{d,1} \\
		0  & 1  & \cdots & s_{d-1,1}  \\ 
		\vdots & \vdots & \ddots & \vdots \\
		0  & 0 & \cdots & 1	 
	\end{pmatrix} ,  
\]
we have
\[
\begin{split}
& \{ P  G \Z^d :  \textnormal{$G$ is lower unitriangular and $P$ is a permutation matrix} \} \\
&= \{ P  G \Z^d :  \textnormal{$G$ is upper unitriangular and $P$ is a permutation matrix} \} \\
&= \{ P  G \Z^d :  \textnormal{$G$ is unitriangular and $P$ is a permutation matrix} \} . 
\end{split}
\] 
Consequently, the term ``lower'' in $\ref{list:cubetiling3}$ of Lemma~\ref{lem:cubetiling} can be replaced with ``upper'' or be removed.   \\
(c) Lemma~\ref{lem:cubetiling} agrees with a result of Kolountzakis \cite[Equation~(8)]{Ko98}, which states that if $(\mathcal{Q}, C\Z^d)$ is a lattice tiling, then after a permutation of the coordinates axes, the matrix $C$ is a lower unitriangular matrix. \\  
(d) Combining Theorem~\ref{thm:main} with Lemma~\ref{lem:cubetiling}, 
we deduce that $(\mathcal{Q}, \Lambda)$ is a tiling with a lattice $\Lambda  \subseteq A \Z^d$ if and only if  
$\Lambda = PG \Z^d$ and $A\Z^d = P G R^{-1} \Z^d$,
where $P \in \{0,1\}^{d \times d}$ is a permutation matrix, $G \in \R^{d \times d}$ is a (lower) unitriangular matrix, and $R \in \Z^{d \times d}$ is a nonsingular integer matrix.   
Note that $\Lambda$ is a subgroup of $A\Z^d$, 
where the density of $\Lambda$ is equal to $|\det (PG) |^{-1} = 1$, and the density of $A\Z^d$ is equal to $|\det (P G R^{-1}) |^{-1} = |\det (R) | =: K \in \N$. Hence, the quotient group $A\Z^d / \Lambda$ has order $K$.  
\end{remark}

\medskip

\begin{proof}[Proof of Lemma~\ref{lem:cubetiling}]
Given a lattice $\Lambda = A \Z^d$ with $A \in \mathrm{GL} (d,\R)$, we may write $\Lambda = \{ n_1 \vec{a}_1 + \cdots + n_d \, \vec{a}_d  \;:\; n_1, \ldots, n_d \in \Z \}$ 
where $\vec{a}_1 , \ldots  , \vec{a}_d$ are the columns of $A$.  
It is also possible that $\Lambda = A' \Z^d$ for some different matrix $A' \in \mathrm{GL} (d,\R)$; in fact, $A'$ can be just a column reordering of $A$ or a completely different matrix.
In view of this, we introduce the following terminology.  
If $\Lambda = [ \, \vec{b}_1 , \ldots  , \vec{b}_d \, ] \, \Z^d$ for some $\vec{b}_1 , \ldots  , \vec{b}_d \in \R^d$, then we call $\mathcal{B} = \{ \vec{b}_1 , \ldots  , \vec{b}_d \}$ a \emph{basis} for $\Lambda$. 
For convenience, we will slightly abuse the notation and also use $\mathcal{B}$ to denote the matrix $[ \, \vec{b}_1 , \ldots  , \vec{b}_d \, ]$. 
We denote by $\{ e_1 , \ldots, e_d \}$ the Euclidean basis of $\R^d$.

$\ref{list:cubetiling1} \Rightarrow \ref{list:cubetiling3}$: \ 
We shall prove by induction that if $\mathcal{Q}^{(d)} = [0,1)^{d}$ is a fundamental domain of $\Lambda$, then $\Lambda=P G \Z^d$ with a permutation matrix $P \in \{0,1\}^{d \times d}$ and a lower unitriangular matrix $G \in \R^{d \times d}$.
To this end, first note that in the case $d=1$ we have $\Lambda=\Z = [1]^{-1} \, [1] \, [1] \, \Z$.  

Let us now assume that the statement holds for $d-1$, and that the $\Lambda$-translates of $\mathcal{Q}^{(d)}$ tile $\R^d$. 
By Theorem~\ref{thm:Minkowski-Hajos}, 
we have $e_j\in \Lambda$ for some $j\in\{1,\ldots,d\}$, so we can find a basis $\mathcal{B}= \{ v_1,\ldots,v_{j-1},e_j,v_{j+1},\ldots,v_{d} \}$ for $\Lambda$, that is, $\Lambda = \mathcal{B} \Z^d$.  
Let $\tau$ be the transposition which interchanges the $j$-th and the $d$-th element, and let $P_\tau$ be the permutation matrix for $\tau$, that is, $P_\tau (j,d) = P_\tau (d,j) = 1$, and $P_\tau (k,k) = 1$ for all $k \neq j, d$, and all other entries are zero, i.e., 
\[
P_\tau
= P_\tau^{-1}
= P_\tau^T 
=
	\begin{pmatrix}
		1& 0 & \cdots & 0 & \cdots & 0 \\ 
		0 & 1 & \cdots & 0 &  \cdots & 0 \\ 
		\vdots & \vdots & \ddots & \vdots &  \ddots & \vdots \\
		0 & 0 & \cdots & 0 & \cdots & 1 \\ 
		\vdots & \vdots & \ddots & \vdots & \ddots & \vdots \\ 
		0 & 0& \cdots & 1 & \cdots & 0
	\end{pmatrix}.
\]
Since $\mathcal{Q}^{(d)}$ is symmetric with respect to all the coordinates, it is also a fundamental domain of  $\Lambda' = \mathcal{B}' \Z^d$ with $\mathcal{B}' := P_\tau \mathcal{B} P_\tau^{-1}$. Note that $\mathcal{B}'$ can be written as $\mathcal{B}'  = \{ v_1' , \ldots, v_{d-1}', e_d \}$ for some $v_1' , \ldots, v_{d-1}' \in \R^d$.  
Let $v_1'' , \ldots,v_{d-1}'' \in \R^{d-1}$ be the vectors obtained by projecting the vectors $v_1' , \ldots, v_{d-1}'$ to the first $d-1$ coordinates, that is, 
\[
\mathcal{B}' = 	\begin{pmatrix}
		|& \cdots & | & 0 \\[2mm] v_1'  & \cdots  & v_{d-1}' & 0  \\[2mm] |& \cdots & | & 0 \\[1mm] | & \cdots & | & 1
	\end{pmatrix} 
= 	\begin{pmatrix}
		|& \cdots & | & 0 \\[2mm] v_1''  & \cdots  & v_{d-1}'' & 0  \\[2mm] |& \cdots & | & 0 \\[1mm] w_1 & \cdots & w_{d-1} & 1
	\end{pmatrix} 	. 
\]
Since $\mathcal{Q}^{(d)}$ a fundamental domain of  $\Lambda'$, it follows that $\mathcal{Q}^{(d-1)} = [0,1)^{d-1}$ is a fundamental domain of $[ v_1'' , \ldots,v_{d-1}'' ] \, \Z^{d-1}$. 
Then the induction hypothesis provides a permutation matrix $P''\in \{0,1\}^{(d-1)\times(d-1)}$ which permutes the elements $1,\ldots,d-1$, and a lower unitriangular matrix $G'' \in \R^{(d-1)\times(d-1)}$ such that $[ v_1'' , \ldots,v_{d-1}'' ] \, \Z^{d-1}= P'' G'' \Z^{d-1} = P'' G'' (P'')^{-1} \Z^{d-1}$.
Setting $w = (w_1, \ldots, w_{d-1})^T$, we have 
\[
\begin{split}
\mathcal{B}' &= 	\begin{pmatrix}
		P'' G'' (P'')^{-1}  & 0 \\ w^T & 1
	\end{pmatrix} \\
	&=
	\underbrace{\begin{pmatrix}
		  P'' & 0 \\ 0 & 1
	\end{pmatrix}}_{=: P' }
	\underbrace{\begin{pmatrix}
		  G'' & 0  \\ w^T P'' & 1
	\end{pmatrix}}_{=: G' }
	\begin{pmatrix}
		 (P'')^{-1}  & 0 \\ 0 & 1
	\end{pmatrix} \\
	&= P'  G'  (P')^{-1} 
\end{split}
\]
and therefore
$\mathcal{B} = P_\tau^{-1} \mathcal{B}' P_\tau = P_\tau^{-1} P'  G'  (P')^{-1}   P_\tau 
= P G P^{-1}$, 
where $G := G'$ is lower unitriangular and $P := P_\tau^{-1} P' \in \{0,1\}^{d \times d}$ is a permutation matrix. 
Hence, we have $\Lambda = \mathcal{B} \Z^d = P G P^{-1} \Z^d = P G \Z^d$, which is $\ref{list:cubetiling3}$. 

$\ref{list:cubetiling3} \Rightarrow \ref{list:cubetiling1}$: \ 
Assume that $\Lambda=P G \Z^d$ with a permutation matrix $P \in \{0,1\}^{d \times d}$ and a lower unitriangular matrix $G \in \R^{d \times d}$.
Since the set of lower unitriangular matrices is a group under matrix multiplication, we may write 
\[
\begin{split}
G^{-1} &= 
	\begin{pmatrix}
		1& 0 & \cdots & 0 \\
		s_{2,1}  & 1  & \cdots & 0  \\ 
		\vdots & \vdots & \ddots & \vdots \\
		s_{d,1}  & s_{d,2} & \cdots & 1	
	\end{pmatrix}  \\
&=
\left(
\begin{array}{c|c|c|c}
e_1 + \sum_{n=2}^d s_{n,1} e_n \,&\,  e_2 + \sum_{n=3}^d s_{n,2} e_n  \,&\,  \cdots  \,&\,  e_d
\end{array}
\right)	
\end{split}
\]
which implies that $G^{-1} \mathcal{Q}$ is a fundamental domain of $\Z^d$. 
Note that $P G^{-1} P^{-1}$ is the matrix obtained by interchanging the coordinates of the matrix $G^{-1}$. 
Therefore, the set $P G^{-1} P^{-1} \mathcal{Q}$ is also a fundamental domain of $\Z^d$, equivalently, the set $\mathcal{Q}$ is a fundamental domain of $(P G^{-1} P^{-1})^{-1} \Z^d = P G P^{-1} \Z^d =P G  \Z^d = \Lambda$. 
\end{proof}

\medskip

We are now ready to prove Theorem~\ref{thm:main}. 

\begin{proof}[Proof of Theorem~\ref{thm:main}]  
$\ref{list:thm1part1} \Rightarrow \ref{list:thm1part2}$: \ 
Assume that $(\mathcal{Q}, \Lambda)$ is a tiling of $\R^d$ with a lattice $\Lambda = C\Z^d \subseteq A\Z^d$ for some $C \in \mathrm{GL} (d,\R)$. Then Lemma~\ref{lem:cubetiling} implies that $P G \Z^d = C\Z^d \subseteq A\Z^d$ with a permutation matrix $P \in \{0,1\}^{d \times d}$ and a lower unitriangular matrix $G \in \R^{d \times d}$. By multiplying both sides with $A^{-1}$, we have $A^{-1} P G \Z^d \subseteq \Z^d$ which implies $A^{-1} P G = R \in \Z^{d\times d}$ and hence, $G = P^{-1} A R$. Since $\widetilde{P} := P^{-1}$ is again a permutation matrix, we conclude that $\widetilde{P} AR$ is unitriangular, that is, $\ref{list:thm1part2}$ holds. 

$\ref{list:thm1part2} \Rightarrow \ref{list:thm1part1}$: \ 
Assume that $G := P A R$ is a unitriangular matrix for a permutation matrix $P \in \{0,1\}^{d \times d}$ and a nonsingular integer matrix $R \in \Z^{d \times d}$. 
Then by setting $\Lambda := P^{-1} G \Z^d = A R \Z^d \subseteq A\Z^d$ and using Lemma~\ref{lem:cubetiling}, we conclude that $(\mathcal{Q}, \Lambda)$ is a tiling of $\R^d$.

$\ref{list:thm1part1} \Rightarrow \ref{list:thm1part3}$: \ 
This is trivial.

$\ref{list:thm1part3} \Rightarrow \ref{list:thm1part1}:$ \ 
If $(\mathcal{Q}, \Phi)$ is a tiling for some $\Phi \subseteq A\Z^d$ with $d\leq 7$, then Theorem~\ref{thm:d-up-to-7-construct-lattice} provides a tiling $(\mathcal{Q}, \Lambda)$ with a lattice $\Lambda \subseteq A\Z^d$. 
\end{proof}

\section{Proof of Theorem~\ref{thm:d-up-to-7-construct-lattice}}
\label{sec:proof-thm-d-up-to-7-construct-lattice}

Let us collect some basic properties of exponential bases.

\begin{lemma}[See e.g., \cite{LPW24-second-RB}]
\label{lem:OB-basic-operations}
Assume that $\mathcal{E}(\Gamma)$ is an orthogonal basis for $L^2(S)$, where $\Gamma \subseteq \R^d$ is a discrete set and $S \subseteq \R^d$ is a measurable set. 
\begin{enumerate}[i.]
    \item For any $a \in \R^d$, the system $\mathcal{E}(\Gamma)$ is an orthogonal basis for $L^2(S+a)$.
    \item For any $b \in \R^d$, the system $\mathcal{E}(\Gamma + b)$ is an orthogonal basis for $L^2(S)$.
    \item For any nonsingular matrix $A \in \R^{d \times d}$, the system $\mathcal{E}(A \Gamma)$ is an orthogonal basis for $L^2(A^{-T} S)$.
In particular, for any $c > 0$, the system $\mathcal{E}(c \Gamma)$ is an orthogonal basis for $L^2(\frac{1}{c} S)$. 
\end{enumerate}
\end{lemma}

For a discrete set $\Gamma \subseteq \R^d$, its lower and upper densities are defined by 
$D^-(\Gamma) := \liminf_{r \rightarrow \infty} \big( \inf_{z \in \R^d} |\Gamma \cap (z + [0,r]^d) | / r^d \big)$ and $D^+(\Gamma) := \limsup_{r \rightarrow \infty} \big( \sup_{z \in \R^d} |\Gamma \cap (z + [0,r]^d) | / r^d \big)$, respectively (see e.g., \cite{He07}). 
If $D^-(\Gamma) = D^+(\Gamma)$, we say that $\Gamma$ has \emph{uniform} density $D(\Gamma) := D^-(\Gamma) = D^+(\Gamma)$. 

\begin{proposition}[Landau's density theorem \cite{La67}] 
\label{prop:Landau}
Let $\Gamma \subseteq \R^d$ be a discrete set and let $S \subseteq \R^d$ be a bounded set.
If $\mathcal{E}(\Gamma)$ is an orthogonal basis for $L^2(S)$, then $\Gamma$ has uniform density and $D(\Gamma) = |S|$.
\end{proposition}

As mentioned in Section~\ref{sec:background}, it holds for any $\Gamma \subseteq \R^d$ that 
$(\mathcal{Q},\Gamma)$ is a spectral pair if and only if 
it is a tiling pair (see Theorem 1.2 in \cite{LRW00}).  
Combining this with Lemma~\ref{lem:OB-basic-operations}, we obtain the following relationship between cube tilings and exponential bases. 

\begin{lemma}
\label{lem:equiv-spectral-tiling}
For any nonsingular matrix $A \in \R^{d \times d}$ and any discrete set $\Gamma \subseteq \R^d$, 
the following are equivalent: 
\begin{enumerate}[i.]

\item
$(\mathcal{Q}, A \Gamma)$ is a tiling pair;

\item
$(\mathcal{Q}, A \Gamma)$ is a spectral pair; 

\item
$(A^T \mathcal{Q},\Gamma)$ is a spectral pair, that is, $\mathcal{E}(\Gamma)$ is an orthogonal basis for $L^2(A^T \mathcal{Q})$. 

\end{enumerate}
\end{lemma}

\begin{remark}\label{rmk:RB-density-Lambda} 
\rm
If $(\mathcal{Q}, \Phi)$ is a tiling of $\R^d$ with some $\Phi \subseteq A \Z^d$, then for $\Gamma := A^{-1} \Phi \subseteq \Z^d$ the system $\mathcal{E}(\Gamma)$ is an orthogonal basis for $L^2(A^T \mathcal{Q})$ by Lemma~\ref{lem:equiv-spectral-tiling}.
Then Proposition~\ref{prop:Landau} gives $D(\Gamma) = | \det A \, |$ but since $\Gamma\subseteq \Z^d$, we necessarily have $| \det A \, | \leq 1$.
In fact, Theorem~\ref{thm:Kolountzakis} asserts that $| \det A \, | = \frac{1}{N}$ for some $N \in \N$.
\end{remark}




We begin with the following criteria for lattice tilings. 

\begin{lemma}\label{lem:equivalencetiling} 
For a lattice $\Lambda=A\Z^d$ with $A \in \mathrm{GL} (d,\R)$, the following are equivalent.
\begin{enumerate}[i.]
     \item\label{list:LatticeTilingEquiv1} 
     $(\mathcal{Q} ,\Lambda)$ is a tiling pair;

     \item\label{list:LatticeTilingEquiv2} 
     $\mathcal{Q}\cap (\mathcal{Q}{+}\lambda) =\emptyset$ for all $0 \neq \lambda\in\Lambda$, and $\Lambda$ has density $1$ (i.e., $| \det A | =1$);

     \item\label{list:LatticeTilingEquiv3} 
     $| \det A | =1$ and $Ak \in (-1,1)^d$ for $k\in\Z^d$ implies $k=0$.
\end{enumerate}
\end{lemma}

\begin{proof}
$\ref{list:LatticeTilingEquiv1} \Rightarrow \ref{list:LatticeTilingEquiv2}$: 
Assume that $(\mathcal{Q} ,\Lambda)$ is a tiling pair, which means that the sets $\mathcal{Q}{+}\lambda$, $\lambda \in \Lambda$, are pairwise disjoint and their union is $\R^d$. 
In particular, we have $\mathcal{Q}\cap (\mathcal{Q}{+}\lambda) =\emptyset$ for all $0 \neq \lambda\in\Lambda$. Since the cube $\mathcal{Q} = [0,1)^d$ has volume $1$ and since $\cupdot_{\lambda \in \Lambda} (\mathcal{Q}{+}\lambda) = \R^d$, the lattice $\Lambda$ must have density $1$. The density of $\Lambda=A\Z^d$ is given by $| \det A |^{-1}$, so we conclude that $| \det A | =1$.

$\ref{list:LatticeTilingEquiv2} \Rightarrow \ref{list:LatticeTilingEquiv3}$: 
Suppose that $Ak \in (-1,1)^d$ for some $k\in\Z^d$. Then $\mathcal{Q} \cap (\mathcal{Q}{+} Ak) \neq \emptyset$, but since $A k \in \Lambda$ we deduce from $\ref{list:LatticeTilingEquiv2}$ that $Ak = 0$. Further, $A \in \mathrm{GL} (d,\R)$ implies that $k= 0 $. The condition $| \det A | =1$ follows immediately from $\ref{list:LatticeTilingEquiv2}$.

$\ref{list:LatticeTilingEquiv3} \Rightarrow \ref{list:LatticeTilingEquiv1}$: 
We will first show that $(\mathcal{Q}{+}\lambda ) \cap ( \mathcal{Q}{+}\lambda' ) = \emptyset$ for every $\lambda \neq \lambda' \in \Lambda$. 
Suppose to the contrary that for some $\lambda \neq \lambda' \in \Lambda$ this does not hold, 
equivalently, $\mathcal{Q} \cap (\mathcal{Q}{+}\lambda' {-} \lambda) \neq \emptyset$. 
This implies that $\lambda' - \lambda \in (-1,1)^d$, but since $\Lambda=A\Z^d$ is a lattice, we may write $\lambda' - \lambda = Ak$ for some $k\in\Z^d$, and thus $Ak \in (-1,1)^d$. 
Now, $\ref{list:LatticeTilingEquiv3}$ implies that $k= 0 $, and in turn, $\lambda' - \lambda = Ak = 0$, yielding a contradiction.

Now, the pairwise disjointness of $\mathcal{Q}{+}\lambda$, $\lambda \in \Lambda$, implies that $\mathcal Q$ is contained in a fundamental domain $\mathcal F$ of the lattice $A\Z^d$. Since $|\det A |=1$, the lattice $\Lambda=A\Z^d$ has density $1$ and every fundamental domain of $\Lambda$ has area $1$. We conclude that $\mathcal F \setminus \mathcal Q$ has Lebesgue measure 0, and $\mathcal Q$ is itself a fundamental domain of $A\Z^d$ in the Lebesgue measure a.e. sense. But clearly, $\R^d \backslash (\mathcal Q{+}\Lambda) =\R^d \backslash ([0,1)^d{+}\Lambda)$ being a set of measure zero implies that it is the empty set.
Hence, we conclude that $(\mathcal{Q} ,\Lambda)$ is a tiling pair. 
\end{proof}

We are now ready to prove Theorem~\ref{thm:d-up-to-7-construct-lattice} by using induction on the dimension $d$. 
For clarity, we use superscript $(d),(d{-}1)$ on all objects to clarify which dimension they belong in. 
Let $\pi_d:\R^d\rightarrow \R^{d-1}$ denote the projection given by $\pi_d(x_1,\ldots,x_{d-1},x_d)=(x_1,\ldots,x_{d-1})$.

\begin{proof}[Proof of Theorem~\ref{thm:d-up-to-7-construct-lattice}]
    For $d=1$, the hypothesis implies that $\Phi=\Z$, so the result holds trivially. In fact, we can choose $v_1=1$ and $w_1=0$.

\vspace{.3cm}
\noindent 
We assume that the result holding in dimension $d-1$  implies that it holds true in  dimension $d$.

Let $(\mathcal{Q}^{(d)},\Phi^{(d)})$ be a tiling pair.
 Since $d\leq 7$, Theorem~\ref{thm:Minkowski-Hajos} implies that there  
 exist $v_d,w_d\in\{0,1\}^d$ so that  $\phi^{(d)}_{v_d},\phi^{(d)}_{w_d}\in\Phi$ are twins, and, after possibly relabeling coordinates,  we assume without loss of generality, that  $\phi^{(d)}_{v_d}-\phi^{(d)}_{w_d}=e_d$. 

Intersecting the sets from the given tiling with the hyperplane $e_d^\perp$, we obtain a tiling $(\mathcal{Q}^{(d-1)},\Phi^{(d-1)})$ with elements in $\Phi^{(d-1)}$ being given by
     $$\phi^{(d-1)}_v=\pi_d \, \phi^{(d)}_{(v,0)}, \quad v\in \Z^{d-1}.$$
     So clearly, we have
     $\phi^{(d-1)}_{(0,\ldots,0)}=(0,\ldots,0)\in\R^{d-1}$.

\def\cube at (#1,#2,#3){\pgfmathsetmacro{\cubex}{1}
\pgfmathsetmacro{\cubey}{1}
\pgfmathsetmacro{\cubez}{-1}
\begin{scope}[shift={(#1,#2,#3)}]

\draw[black] (0,0,\cubez) -- ++(\cubex,0,0) -- ++(0,\cubey,0) -- ++(-\cubex,0,0) -- cycle;
\draw[black] (0,0,\cubez)-- ++(\cubex,0,0)-- ++(0,0,-\cubez) -- ++(-\cubex,0,0) -- cycle;
\draw[black] (0,0,\cubez) -- ++(0,\cubey,0) -- ++(0,0,-\cubez) -- ++(0,-\cubey,0) -- cycle;

\draw[black,fill=white!80] (\cubex,\cubey,0) -- ++(-\cubex,0,0) -- ++(0,-\cubey,0) -- ++(\cubex,0,0) -- cycle;
\draw[black,fill=white!80] (\cubex,\cubey,0)-- ++(-\cubex,0,0)-- ++(0,0,\cubez) -- ++(\cubex,0,0) -- cycle;
\draw[black,fill=white!80] (\cubex,\cubey,0) -- ++(0,-\cubey,0) -- ++(0,0,\cubez) -- ++(0,\cubey,0) -- cycle;

\end{scope}}
\def\cubeintersect at (#1,#2,#3){\pgfmathsetmacro{\cubex}{1}
\pgfmathsetmacro{\cubey}{1}
\pgfmathsetmacro{\cubez}{-1}
\begin{scope}[shift={(#1,0,#3)}]

\draw[blue] (0,0,\cubez)-- ++(\cubex,0,0)-- ++(0,0,-\cubez) -- ++(-\cubex,0,0) -- cycle;


\end{scope}}
\begin{figure}
\begin{center}
\begin{tikzpicture}[fill opacity=0.8, scale=2]

\cube at (0,-.5,-1);  
\cube at (0,0,0);
\cube at (1,-.3,-1+.6); 
\cube at (1,-.6,.6);
\cube at (2,-.1,-1+.2);
\cube at (2,-.5,.2); 
\cubeintersect at (0,-.5,-1);  
\cubeintersect at (0,0,0);
\cubeintersect at (1,-.3,-1+.6); 
\cubeintersect at (1,-.6,.6);
\cubeintersect at (2,-.1,-1+.2);
\cubeintersect at (2,-.5,.2); 

\fill[fill=gray, opacity=0.2] (-.5,0,1.2) -- (3.5,0,1.2) -- (3.5,0,-2.5) -- (-.5,0,-2.5) -- cycle;
\node[above] at (-.5,0,1.2) { $e_3^\perp$};



\fill[red] (0,0,0) circle (0.2mm); \node[above] at (0.1,-.19,0) {\tiny (\!0\!,0\!,0\!)};
\fill[red] (1,0,0) circle (0.2mm); \node[above] at (1.1,-.19,0) {\tiny (\!1\!,0\!,0\!)};
\fill[red] (2,0,0) circle (0.2mm); \node[above] at (2.1,-.19,0) {\tiny (\!2\!,0\!,0\!)};
\fill[red] (0,0,-1) circle (0.2mm); \node[above] at (0.1,-.19,-1) {\tiny (\!0\!,1\!,0\!)};
\fill[red] (1,0,-1) circle (0.2mm); \node[above] at (1.1,-.19,-1) {\tiny (\!1\!,1\!,0\!)};
\fill[red] (2,0,-1) circle (0.2mm); \node[above] at (2.1,-.19,-1) {\tiny (\!2\!,1\!,0\!)};


\end{tikzpicture}
    \end{center}
    \caption{
    Intersection of a cube tiling in $\R^3$ with the plane $e_3^\perp$, that is, with orthogonal complement of $e^3$.   The intersection gives a tiling of the $2$-dimensional plane, that is, of $\R^2$.
}
\end{figure}
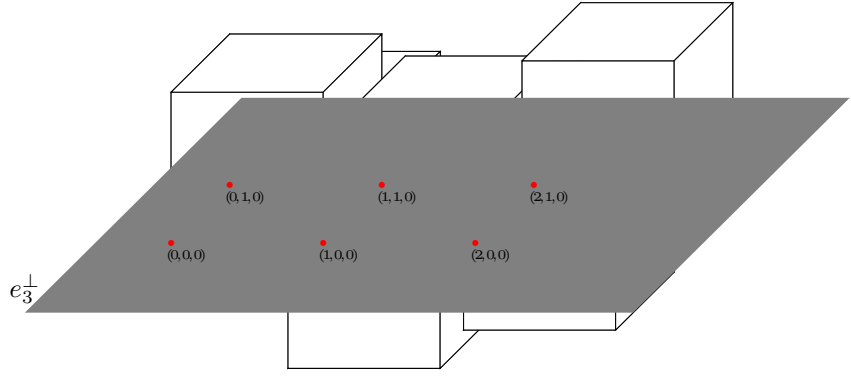

By induction hypothesis, there exist $p_1,\ldots,p_{d-1},q_1,\ldots,q_{d-1}\in \{0,1\}^{d-1}$ such that
\begin{align*}
    \Big(\mathcal{Q}^{(d-1)},\big(\phi^{(d-1)}_{p_1}-\phi^{(d-1)}_{q_1}\big)\Z + \ldots + (\phi^{(d-1)}_{p_{d-1}}-\phi^{(d-1)}_{q_{d-1}}\big)\Z\Big)
\end{align*} 
is a tiling pair for $\R^{d-1}$. 
Then Lemma~\ref{lem:equivalencetiling} guarantees that the matrix  
\[
M^{(d-1)} 
=
    \begin{pmatrix}
        \phi^{(d-1)}_{p_1}-\phi^{(d-1)}_{q_1} & \big| &  \cdots & \big| & \phi^{(d-1)}_{p_{d-1}}-\phi^{(d-1)}_{q_{d-1}}  
    \end{pmatrix}
\]
satisfies $|\det M^{(d-1)}| =1$
and that $M^{(d-1)}k\in(-1,1)^d$ for $k\in\Z^{d-1}$ implies $k= 0 $.

We claim that 
\begin{align*}
    \Big(\mathcal{Q}^{(d)},\big(\phi^{(d)}_{(p_1,0)}-\phi^{(d)}_{(q_1,0)}\big)\Z + \ldots + 
    \big(\phi^{(d)}_{(p_{d-1},0)}-\phi^{(d)}_{(q_{d-1},0)}\big)\Z+e_d\Z) 
\end{align*}
is a tiling pair for $\R^d$. Let
\begin{align*}
   M^{(d)}&= \begin{pmatrix}
        \phi^{(d)}_{(p_1,0)}-\phi^{(d)}_{(q_1,0)} & \big| &   \cdots & \big| & \phi^{(d)}_{(p_{d-1},0)}-\phi^{(d)}_{(q_{d-1},0)} & \big| & e_d  
    \end{pmatrix}\\
    &=
    \begin{pmatrix}
         \phi^{(d-1)}_{p_1}-\phi^{(d-1)}_{q_1} & \big| &   \cdots & \big| & \phi^{(d-1)}_{p_{d-1}}-\phi^{(d-1)}_{q_{d-1}}  & \big| &  0   \\[1mm]
        \hline\\[-3.5mm]
        a_1 & \big| &   \cdots & \big| & a_{d-1} & \big| & 1 
    \end{pmatrix}
    \\
    &=
    \begin{pmatrix}
         &  &   M^{(d-1)} &    & \big|&  0  \\
        \hline
         & a_1 &   \cdots &   a_{d-1} & \big| & 1 
    \end{pmatrix}
\end{align*}
where $a_i\in \R$ is the $d$-th component of $\phi^{(d)}_{(p_i,0)}-\phi^{(d)}_{(q_i,0)}$ for $i=1,\ldots,d$. Clearly, we have $|\det M^{(d)}| = 1\cdot |\det M^{(d-1)}| =1$.
Further, the condition $M^{(d)}(k_1,\ldots,k_d)^T\in (-1,1)^d$ is equivalent to having $M^{(d-1)}(k_1,\ldots,k_{d-1})^T \in (-1,1)^{d-1}$ and $k_1a_1+\ldots + k_{d-1}a_{d-1}+k_d\in (-1,1)$. Now, $M^{(d-1)}(k_1,\ldots,k_{d-1})^T \in (-1,1)^{d-1}$  implies $k_1=\ldots=k_{d-1}=0$, and, therefore, we have $k_d=k_1a_1+\ldots + k_{d-1}a_{d-1}+k_d\in (-1,1)$ and $k_d=0$ as well. The claim then follows from applying again Lemma~\ref{lem:equivalencetiling}.

Setting $v_i=(p_i,0)\in \{0,1\}^d$ and $w_i=(q_i,0)\in \{0,1\}^d$ for $i=1,\ldots,d$ completes the induction step.   
\end{proof}

\section{Proof of Theorem~\ref{thm:Kolountzakis}}\label{sec:Kolountzakis}

With permission of Mihalis Kolountzakis we include his proof of Theorem~\ref{thm:Kolountzakis}. Here, we will assume only that $(S,\Phi)$ is 
a tiling pair for $\R^d$ a.e., in the sense that the sets $\R^d \setminus \bigcup_{\varphi \in \Phi} S {+} \varphi$ and $( S{+}\varphi ) \cap ( S{+}\varphi' )$ for  $\varphi\neq \varphi'$ in $\Phi$ are of Lebesgue measure 0, so not necessarily empty.  

\begin{proof}[Proof of Theorem~\ref{thm:Kolountzakis}]
The tiling condition of $(S,\Phi)$ implies $f(x)=\sum_{\varphi\in\Phi} \mathbf{1}_S (x-\varphi)=1$ for almost every $x\in\R^d$. Possibly replacing $S$ with one of its translates allows us to assume $f(n)=1$ for all $n\in\Z^d$.

Let $K=S\cap \Z^d\subseteq \Z^d$. As $\Phi\subseteq \Z^d$ we have $K+\varphi\subseteq \Z^d$ for all $\varphi\in\Phi$. Moreover, 
$\sum_{\varphi\in\Phi} \mathbf{1}_K (n-\varphi)=f(n)=1$ for $n\in\Z^d$ implies that $(K,\Phi)$ is a tiling pair for $\Z^d$. This clearly implies that the density of $\Phi$ is given by $D(\Phi)=\frac 1 {|K|}$ where $|K|$ is the cardinality of the set $K$. As $(S,\Phi)$ being a tiling pair implies $|S| = D(\Phi)^{-1}$, the first statement follows.

To prove the second part of Theorem~\ref{thm:Kolountzakis}, assume that $(\mathcal Q,\Phi)$ is a tiling pair with $\Phi\subseteq A\Z^d$. Then $(A^{-1} \mathcal Q, A^{-1}\Phi)$ is also a tiling pair with $A^{-1}\Phi \subseteq \Z^d$, and hence the first statement implies $|\det A|^{-1}= |(\det A)^{-1}|=|A^{-1} \mathcal{Q} |\in\N$. 
\end{proof}

\section{Proof of Corollary~\ref{cor:main-application-to-expRB}}
\label{sec:proof-cor-main-application-to-expRB}


$\ref{list:thm1part12} \Leftrightarrow \ref{list:thm1part22}$: \ 
Theorem~\ref{thm:main} implies that condition ({\romannumeral 2}) is equivalent to the existence of a lattice $\Lambda \subseteq B^T A\Z^d$ with $(\mathcal{Q}, \Lambda)$ being a tiling, that is, $\mathcal{E}(\Lambda)$ being an orthogonal basis for $L^2(\mathcal{Q})$ (see Lemma~\ref{lem:equiv-spectral-tiling}).
By setting $\widetilde{\Lambda} := B^{-T} \Lambda \subseteq A\Z^d$ and using Lemma~\ref{lem:OB-basic-operations}, this is again equivalent to $\mathcal{E}(\widetilde{\Lambda})$ being an orthogonal basis for $L^2(B \mathcal{Q})$, which is $\ref{list:thm1part12}$. 

$\ref{list:thm1part22} \Leftrightarrow \ref{list:thm1part22-prime} {}^{\prime}$: \ 
Assume that $\ref{list:thm1part22}$ holds and let $G := P B^T AR$ be the resulting unitriangular matrix. 
Then $B = A^{-T} R^{-T} G^T P^{-T}$ and therefore, $B\mathcal{Q} = A^{-T} \widetilde{R}^{-1} \widetilde{G} \mathcal{Q}$ where $\widetilde{R} := R^T \in \Z^{d \times d}$ is again a nonsingular integer matrix and $\widetilde{G} := G^T$ is again a unitriangular matrix, which gives $\ref{list:thm1part22-prime} {}^{\prime}$. 
Conversely, assume that $\ref{list:thm1part22-prime} {}^{\prime}$ holds and let $B\mathcal{Q} = A^{-T} R^{-1} G \mathcal{Q}$ with a nonsingular integer matrix $R \in \Z^{d \times d}$ and a unitriangular matrix $G \in \R^{d \times d}$. Then $B P = A^{-T} R^{-1} G$ with a permutation matrix $P \in \{0,1\}^{d \times d}$, and therefore, by setting $\widetilde{P} := P^T$ and $\widetilde{R} := R^T$, we have that $\widetilde{P} B^T A \widetilde{R} = P^T B^T A R^T = G^T$ is unitriangular, which gives $\ref{list:thm1part22}$. 

$\ref{list:thm1part12} \Leftrightarrow \ref{list:thm1part32}$: \ 
By Lemma~\ref{lem:OB-basic-operations}, condition ${\romannumeral 3}$ is equivalent to the existence of a set $\Phi \subseteq A\Z^d$ with
$\mathcal{E}( B^T \Phi)$ being an orthogonal basis for $L^2(\mathcal{Q})$,
that is, $(\mathcal{Q}, B^T \Phi)$ being a tiling (see Lemma~\ref{lem:equiv-spectral-tiling}), where $B^T \Phi \subseteq B^T A\Z^d$. 
If $d\leq 7$, then Theorem~\ref{thm:main} shows that this is equivalent to the existence of a tiling $(\mathcal{Q}, \Lambda)$ with a lattice $\Lambda \subseteq B^T A\Z^d$. 
Further, by setting $\widetilde{\Lambda} := B^{-T} \Lambda \subseteq A\Z^d$ and using Lemma~\ref{lem:OB-basic-operations}, this is again equivalent to $\mathcal{E}(\widetilde{\Lambda})$ being an orthogonal basis for $L^2(B \mathcal{Q})$, which is $\ref{list:thm1part12}$.
\hfill $\Box$ 

\backmatter

%
%
%

\bmhead{Acknowledgments}

D.G.~Lee is supported by the National Research Foundation of Korea (NRF) grant funded by the Korean government (MSIT) (RS-2023-00275360).
G.E.~Pfander acknowledges support by the German Research Foundation (DFG) Grant PF 450/11-1.

%

\section*{Declarations}

\subsection*{Funding}
This work was supported by the National Research Foundation of Korea (NRF) grant funded by the Korean government (MSIT) (RS-2023-00275360), and the German Research Foundation (DFG) Grant PF 450/11-1.

\subsection*{Conflict of interest}

On behalf of all authors, the corresponding author states that there is no conflict of interest.

\subsection*{Availability of data and materials}

Not applicable. 

\subsection*{Authors' contributions}
All authors contributed equally to this work. All authors are first authors, and the authors' names are listed in alphabetical order.

\bibliography{submission-arXiv/parallelepipeds-bibliography}

\end{document}